\documentclass{amsart}

\newcounter{commentcounter}

\usepackage{amsmath} 
\usepackage{amssymb} 
\usepackage{amsthm} 
\usepackage{stmaryrd} 
\usepackage[english]{babel} 
\usepackage[font=small,justification=centering]{caption} 
\usepackage[nodayofweek]{datetime}
\usepackage[shortlabels]{enumitem} 
\usepackage[T1]{fontenc} 
\usepackage[utf8]{inputenc} 
\usepackage{ifthen} 
\usepackage{mathabx} 
\usepackage{mathtools} 
\usepackage[dvipsnames]{xcolor} 
\usepackage[pdftex,  colorlinks=true, pagebackref=true]{hyperref} 
    \hypersetup{urlcolor=RoyalBlue, linkcolor=RoyalBlue,  citecolor=black}
\usepackage{setspace} 
\usepackage{tikz-cd} 
\usepackage{xfrac} 
\usepackage[capitalize]{cleveref} 

\renewcommand*{\backref}[1]{}
\renewcommand*{\backrefalt}[4]
{
    \ifcase #1
        No citation in the text.
    \or
        Cited on Page #2.
    \else
        Cited on Pages #2.
    \fi
}

\makeatletter
\@namedef{subjclassname@1991}{Mathematical subject classification 1991}
\@namedef{subjclassname@2000}{Mathematical subject classification 2000}
\@namedef{subjclassname@2010}{Mathematical subject classification 2010}
\@namedef{subjclassname@2020}{Mathematical subject classification 2020}
\makeatother

\newtheorem{thm}{Theorem}[section]
\newtheorem{lemma}[thm]{Lemma}

\newtheorem{prop}[thm]{Proposition}
\newtheorem{conjecture}[thm]{Conjecture}
\newtheorem{claim}[thm]{Claim}

\newtheorem{question}[thm]{Question}

\newtheorem{thmx}{Theorem}
\newtheorem{corx}[thmx]{Corollary}

\theoremstyle{definition}

\theoremstyle{definition}

\newtheorem{remark}[thm]{Remark}

\theoremstyle{plain}

    \newtheoremstyle{TheoremNum}
        {8.0pt plus 2.0pt minus 4.0pt}{8.0pt plus 2.0pt minus 4.0pt} 
        {\itshape} 
        {-0.15cm} 
        {\bfseries} 
        {.} 
        { }  
        {\thmname{#1}\thmnote{ \bfseries #3}}
    \theoremstyle{TheoremNum}
    \newtheorem{duplicate}{}

\newcommand*{\claimproofname}{My proof}
\newenvironment{claimproof}[1][\claimproofname]{\begin{proof}[#1]}{\end{proof}}

\DeclareMathOperator{\aster}{\text{\LARGE{\textasteriskcentered}}}

\DeclareMathOperator{\Aut}{\mathrm{Aut}}

\DeclareMathOperator{\Comm}{\mathrm{Comm}}
\DeclareMathOperator{\Isom}{\mathrm{Isom}}

\newcommand{\calg}{{\mathcal{G}}}

\newcommand{\calt}{{\mathcal{T}}}

\newcommand{\calv}{{\mathcal{V}}}

\newcommand{\res}{{\rm res}}

\newcommand{\GL}{\mathrm{GL}}

\newcommand{\PSL}{\mathrm{PSL}}

\newcommand{\SU}{\mathrm{SU}}

\newcommand{\OO}{\mathrm{O}}
\newcommand{\SO}{\mathrm{SO}}

\newcommand{\LM}{\mathrm{LM}}

\newcommand{\CAT}{\mathrm{CAT}}

\newcommand{\onto}{\twoheadrightarrow}

\DeclareMathOperator{\Lk}{\mathrm{Lk}}

\newcommand{\EE}{\mathbb{E}}

\newcommand{\cd}{\mathrm{cd}}

\def\Z{\mathbb{Z}}

\newcommand{\ZZ}{\mathbb{Z}}

\newcommand{\RR}{\mathbb{R}}
\newcommand{\QQ}{\mathbb{Q}}

\usepackage{tikz}
\usetikzlibrary{arrows,quotes}
\tikzstyle{blackNode}=[fill=black, draw=black, shape=circle]

\usepackage[foot]{amsaddr}
\title{Irreducible lattices fibring over the circle}
\author{Sam Hughes}
\address{Sam Hughes\\ Universit\"at Bonn, Mathematical Institute, Endenicher Allee 60, 53115 Bonn, Germany}
\email{sam.hughes.maths@gmail.com} \email{hughes@uni-bonn.de}
\date{\today}

\begin{document}

\begin{abstract}
We investigate the Bieri--Neumann--Strebel--Renz (BNSR) invariants of irreducible uniform lattices. In the case of a direct product of a tree and a Euclidean space we show that vanishing of the BNSR invariants for all finite-index subgroups of a given uniform lattice is equivalent to irreducibility.  On the other hand we construct irreducible uniform lattices which admit maps to the integers whose kernels' finiteness properties are determined by the finiteness properties of certain Bestvina--Brady groups.
\end{abstract}

\maketitle

\section{Introduction}
Let $H$ be a locally compact group with Haar measure $\mu$. A \emph{lattice} $\Gamma$ in $H$ is a discrete subgroup such that $H/\Gamma$ has finite measure.  We say $\Gamma$ is \emph{uniform} if $H/\Gamma$ is compact.  Roughly speaking, a lattice $\Gamma$ in a product $G\times H$ is \emph{irreducible} if the projections of $\Gamma$ to $G$ and $H$ are non-discrete and $\Gamma$ does not virtually split as a direct product of two infinite groups, otherwise we say $\Gamma$ is \emph{reducible} (we will give the precise definition in \Cref{sec.irr}).  A celebrated application of Margulis's normal subgroup theorem \cite{Margulis1978} connects, in the case of lattices in semsimple Lie groups, irreducibility with vanishing of the first cohomology group.

\begin{thm}[Margulis]\label{thm.Margulis}
Let $\Gamma$ be a lattice in semisimple Lie group with finite centre and real rank at least $2$. If $H^1(\Gamma)\neq0$, then $\Gamma$ virtually splits as a direct product of two infinite groups.
\end{thm}

We will now broaden our scope to lattices in products of isometry groups of irreducible minimal $\CAT(0)$ spaces.  Here a $\CAT(0)$ space $X$ is \emph{irreducible} if $X$ does not split as a direct product of two subspaces and is \emph{minimal} if there is no $\Isom(X)$-invariant closed convex non-empty proper subspace $X'\subset X$.  In this later case we say that $\Isom(X)$ acts \emph{minimally}. The reader can consult \cite{BridsonHaefliger1999} for a comprehensive introduction to the theory of $\CAT(0)$ spaces and \cite{CapraceMonod2009a,CapraceMonod2009b,CapraceMonod2019} for a structure theory of the spaces and their isometry groups.  

In this more general setting the universal covering trick of Burger--Mozes shows that a generalisation of \Cref{thm.Margulis} even to lattices in products of trees and symmetric spaces fails (see \cite{BurgerMozes2000a}).  However, if the first cohomology group is non-zero we are able to deploy secondary invariants introduced in \cite{BieriNeumannStrebel1987,BieriRenz1988} called \emph{BNSR} or \emph{$\Sigma$-invariants} $\Sigma^n(\Gamma)$ and $\Sigma^n(\Gamma;\ZZ)$ which measure how far a first cohomology class is from a fibration $B\Gamma\to S^1$ of \emph{finite} CW complexes.    For an explicit example of an irreducible lattice with $H^1(\Gamma;\RR)\neq 0$, the reader is referred to \Cref{sec.LM}.

A first cohomology class $\varphi$ and its inverse $-\varphi$ are in $\Sigma^n(\Gamma)$ (resp. $\Sigma^n(\Gamma;\ZZ)$) if and only if $\varphi$ is $\mathsf{F}_n$-fibred (resp. $\mathsf{FP}_n$-fibred).  Here, $\varphi$ is $\mathsf{F}_n$-fibred (resp. $\mathsf{FP}_n$-fibred) if $\ker(\varphi)$ is type $\mathsf{F}_n$, that is, there exists a model for $K(\ker(\varphi),1)$ with finite $n$-skeleton (resp. type $\mathsf{FP}_n$, that is, there exists a projective resolution $P_\ast\to\ZZ$ over $\ZZ[\ker(\varphi)]$ such that for each $i\leq n$ the module $P_i$ is a finitely generated $\ZZ[\ker(\varphi)]$-module).  If $G$ is $\mathsf{F}_1$-fibred we may say $G$ is \emph{algebraically fibred}.  Motivated by this we ask the following question and answer it in several cases.

\begin{question}\label{Q.motivation}
Let $\Gamma$ be a uniform lattice in a product $X_1\times X_2$ of proper minimal unbounded $\CAT(0)$ spaces.  If $\Sigma^n(\Gamma)$ or $\Sigma^n(\Gamma;\ZZ)$ is non-empty for some $n\geq1$, then is $\Gamma$ necessarily reducible?
\end{question}

Note that \Cref{Q.motivation} appears to be open even in the case of a product of trees.  There are plenty of irreducible $\CAT(0)$ groups which virtually fibre - we will explain how these either give positive answers to \Cref{Q.motivation} or are not within its remit.  In the seminal work of Bestvina and Brady \cite{BestvinaBrady1997}, the authors show that there exist characters of right angled Artin groups (RAAGs) which $\mathsf{FP}_2$-fibre but not $\mathsf{F}_2$-fibre.  We mention here that every RAAG is either a direct product of two infinite subgroups or is a lattice in a single irreducible $\CAT(0)$ space.  Generalisations to obtain uncountably many (quasi-isometry classes of) groups of type $\mathsf{FP}$ have been considered by Leary \cite{Leary2018a} (Kropholler--Leary--Soroko \cite{KrophollerLearySoroko2020}) and Brown--Leary \cite{BrownLeary2020}.  For right angled Coxeter groups (RACGs) there is work of Jankiewicz--Norin--Wise \cite{JankiewiczNorinWise2021} where the authors algebraically fibre certain finite index subgroups and work of Schesler--Zaremsky \cite{ScheslerZaremsky2021} where the authors take a probabilistic viewpoint.  As in the case of RAAGs every RACG is either a direct product of two infinite subgroups or is a lattice in a single irreducible $\CAT(0)$ space.

A deep theorem of Agol states that hyperbolic $3$-manifolds virtually fibre \cite{Agol2013}.  We briefly mention that this result has been generalised to the setting of RFRS groups by Kielak \cite{Kielak2020} and improved further by Fisher \cite{Fisher2021}.  The relationship to homology growth has been explored in \cite{FisherHughesLeary2023} and to profinite rigidity in \cite{HughesKielak2022}. In higher dimensions a number of hyperbolic $n$-manifolds have been algebraically fibred in the work of Battista, Isenrich, Italiano, Martelli, Migliorini, and Py \cite{BattistaMartelli2021,ItalianoMartelliMigliorini2021a,IsenrichMartelliPy2021}.  We highlight the paper of Italiano--Martelli--Migliorini \cite{ItalianoMartelliMigliorini2021b} where the authors fibre a hyperbolic $5$-manifold over $S^1$.  Of course in every case each group is a lattice in a single irreducible $\CAT(0)$ space.

In the case of a uniform lattice in the product of a locally-finite tree and a Euclidean space we give a positive answer to Question~\ref{Q.motivation}.  For an explicit example of an irreducible lattice in such a product with non-trivial first cohomology the reader is referred to \Cref{sec.LM}.  The existence of irreducible lattices was demonstrated by Leary and Minasyan - where they construct the first examples of $\CAT(0)$ but not biautomatic groups \cite{LearyMinasyan2019}; a rough classification of such lattices was obtained by the author in \cite{Hughes2021a}.  Note that the following theorem is new even for Leary--Minasyan groups.  

\begin{duplicate}[\Cref{thmx.fibring}]
Let $\calt$ be a locally-finite leafless cocompact tree, not isometric to $\RR$, and let $T=\Aut(\calt)$.  Let $\Gamma$ be a uniform $(\Isom(\EE^n)\times T)$-lattice, then $\Gamma$ virtually algebraically fibres if and only if $\Gamma$ is reducible.
\end{duplicate}

A group $\Gamma$ \emph{virtually fibres (over the circle)} if there exists a finite-index subgroup $\Gamma'\leq\Gamma$ and a character $\varphi\in H^1(\Gamma';\RR)$ such that $\ker(\varphi)$ is of type $\mathsf{F}$, that is, there exists a finite model for $K(\ker(\varphi),1)$.

\begin{duplicate}[\Cref{corx.fibring.dim2}]
With notation as in Theorem~\ref{thmx.fibring}, suppose $n=2$.  Then, $\Gamma$ virtually fibres if and only if $\Gamma$ is reducible.
\end{duplicate}

The main obstruction to extending the previous corollary to higher dimensional Euclidean spaces (i.e. $n\geq3$) is that we do not know if every $(\Isom(\EE^{n-1})\times T)$-lattice is virtually torsion-free (see \cite[Question~9.1]{Hughes2021a}).

On the other hand we study lattices in products of symmetric spaces with the universal cover of a Salvetti complex and find many examples of lattices which fibre over the circle.

\begin{duplicate}[\Cref{thmx.C}]
    There exists an irreducible lattice fibring over the circle.
\end{duplicate}

More precisely, we give a family of examples of irreducible uniform $\PSL_2(\RR)\times S_L$ lattices which fibre over the circle.  Here $S_L=\Aut(\widetilde X_L)$ where $X_L$ is the Salvetti complex associated with the right-angled Artin group $A_L$. Note the construction is much more general.  The reader is referred to Sections \ref{extending actions} and \ref{sec.irr.fibring} for more details.

\subsection{Conjectures and questions}
One by product of the proof of \Cref{thmx.fibring} is the following computation of the first cohomology for a large family of lattices.

\begin{duplicate}[\Cref{prop.isomEnxX.H1}]
Let $X$ be an irreducible locally finite $\CAT(0)$ polyhedral complex, let $A=\Aut(X)$ act cocompactly and minimally, and let $\Gamma$ be a uniform $(\Isom(\EE^n)\times A)$-lattice.  If $\Gamma$ is irreducible, then $H^1(\Gamma;\ZZ)\cong H^1(X/\Gamma;\ZZ)$.
\end{duplicate}

The author suspects that this phenomena is much more general and conjectures the following.

\begin{conjecture}
    Let $X$ be an irreducible locally finite $\CAT(0)$ polyhedral complex, let $A=\Aut(X)$ act cocompactly and minimally, let $Y$ be a symmetric space of non-compact type with corresponding semi-simple Lie group $G$, and let $\Gamma$ be a uniform $(G\times \Isom(\EE^n)\times A)$-lattice.  If $\Gamma$ is irreducible, then $H^1(\Gamma;\ZZ)\cong H^1(X/\Gamma;\ZZ)$.
\end{conjecture}

Using \Cref{claim.2} one can discount the Euclidean factor.  Moreover, one can easily reduce to the case where no factor of $G$ has Kazhdan's property (T).  This essentially leaves the case where $G$ is a product of $\SO(m,1)$ and $\SU(n,1)$ for various $m$ and $n$.

Let $L$ be a flag complex and suppose $L$ is not connected.  It is well known that no character of the right-angled Artin group $A_L$ is algebraically fibred.  We suspect this behaviour holds for all irreducible uniform $(\Isom(\EE^n)\times S_L)$-lattices.

\begin{question}
Let $L$ be a flag complex.  Is it true that if $L$ is not connected, then no irreducible uniform $(\Isom(\EE^n)\times S_L)$-lattice is algebraically fibred?
\end{question}


\subsection{Acknowledgements}
This paper contains material from the author's PhD thesis \cite{HughesThesis}.  The author would like to thank his PhD supervisor Ian Leary for his guidance and support.  The author would like to thank Dawid Kielak for carefully reading an earlier draft of this paper.  The author would also like to thank Pierre-Emmanuel Caprace and Jingyin Huang for helpful correspondence.  This work was supported by the Engineering and Physical Sciences Research Council grant number 2127970.  This work has received funding from the European Research Council (ERC) under the European Union’s Horizon 2020 research and innovation programme (Grant agreement No. 850930).  The author would also like to thank an anonymous referee for catching a fatal mistake in an earlier draft of this paper.  The author would like to thank both referees for numerous helpful comments which dramatically improved the exposition of this paper.

\subsection{Bibliographical note}
 This is the second part of a longer paper contained in the author's PhD thesis which was split at the request of a referee (see \cite[Paper~4]{HughesThesis}).  The first part of this longer paper can be found in \cite{Hughes2021a}.  Note that some of the results here are not contained in the author's PhD thesis.  Also note that a number of group presentations and results regarding residual finiteness and autostackability will only exist in the thesis version.  Finally, we remark the existence of the companion papers \cite{Hughes2021b,Hughes2023hdQ,HughesValiunas2022,HughesValiunas2023} where similar techniques are used to prove a number of other results.

Since writing this paper and fixing the error caught by one of the anonymous referees, Jingyin Huang and Mahan Mj \cite{HuangMj2023} (as well as Camille Horbez and Jingyin Huang \cite{HorbezHuang2023})  have considered related constructions in the context of commensurators (and measure equivalence).  Their construction fixes the major oversight with the construction in an earlier version of this paper.  It remains unclear to the author when there exist irreducible uniform lattices in products of Salvetti complexes and Euclidean spaces.

\section{Preliminaries}\label{sec.prelims}

\subsection{Lattices}
Let $H$ be a locally compact topological group with right invariant Haar measure $\mu$.  A discrete subgroup $\Gamma\leq H$ is a \emph{lattice} if the covolume $\mu(H/\Gamma)$ is finite.  A lattice is \emph{uniform} if $H/\Gamma$ is compact and \emph{non-uniform} otherwise.  Let $S$ be a right $H$-set such that for all $s\in S$, the stabilisers $H_s$ are compact and open.  Then, if $\Gamma\leq H$ is discrete, the stabilisers of $\Gamma$ acting on $S$ are finite.

Let $X$ be a locally finite, connected, simply connected simplicial complex. The group $H=\Aut(X)$ of simplicial automorphisms of $X$ naturally has the structure of a locally compact topological group, where the topology is given by uniform convergence on compacta.

Note that $T$, the automorphism group of a locally-finite tree $\calt$, admits lattices if and only if the group $T$ is unimodular (that is, the left and right Haar measures coincide).  In this case we say $\calt$ is \emph{unimodular}.  We say a tree $\calt$ is \emph{leafless} if it has no vertices of valence one.

\subsection{Irreducibility}\label{sec.irr}
Two notions of irreducibility will feature in this paper; for uniform $\CAT(0)$ lattices they are equivalent due to a theorem of Caprace--Monod.  See \cite[Section 2.3]{Hughes2021a} for an extended discussion concerning these definitions.

Let $X=\EE^n\times X_1\times\dots\times X_m$ be a product of irreducible proper $\CAT(0)$ spaces with each $X_i$ not quasi-isometric to $\EE^1$ and let $H= H_0\times H_1\times\dots\times H_m\coloneqq \Isom(\EE^n)\times\Isom(X_1)\times\dots\times\Isom(X_m)$, such that for each $i\geq 1$ the group $H_i$ is non-discrete, cocompact, and acting minimally on $X_i$.  Let $\Gamma$ be a uniform lattice in $H$.  Note that by \cite[Addendum~1.8]{CapraceMonod2009a} and \cite[Corollary~3.12]{CapraceMonod2009b}, since $\Gamma$ is finitely generated, the product decomposition of $X$ above is unique.

We have projections $\pi_i\colon \Gamma \to H_i$ for each factor $H_i$.  The Euclidean factor gives us two further projections, firstly, $\pi_{\Isom(\EE^n)}\colon \Gamma\to\Isom(\EE^n)\cong \RR^n\rtimes \OO(n)$ and secondly, $\pi_{\OO(n)}$ which is defined as the composition $\Gamma\to\Isom(\EE^n)\to\OO(n)$.

Suppose $n=0$, then we say $\Gamma$ is \emph{weakly irreducible} if the projection of $\Gamma$ to each proper subproduct $H_I\coloneqq \prod_{i\in I}H_i$ for each proper subset $I\subset\{1,\dots,m\}$ is non-discrete. 

Suppose $n=1$, then $\Gamma$ always virtually splits a direct product $\Gamma'\times \Z$ by \cite{CapraceMonod2019}. In this case we always define $\Gamma$ to be \emph{reducible}.  

Suppose $n\geq 2$.  Let $\ell$ be the maximal integer such that $\EE^n=\prod_{j=1}^\ell \EE^{k_j}$ with each $k_j\geq 1$ such that the product decomposition is preserved by \emph{some} finite index subgroup $\Lambda$ of $\Gamma$.  Observe that $\Lambda$ is contained in $\prod_{j=1}^\ell\Isom(\EE^{k_j})\times\prod_{i=1}^m H_i$.  Denote each $\Isom(\EE^{k_j})$ by $E_j$ and the corresponding orthogonal group by $O_j$.  Then for $\Gamma$ to be \emph{weakly irreducible} we require that each $k_j\geq2$, and that the projection $\pi_{I,J}$ of $\Lambda$ to each proper subproduct, $G_{I,J}\coloneqq \prod_{j\in J}E_j\times\prod_{i\in I}H_i$ for $I\subseteq\{1,\dots,m\}$ and $J\subseteq\{1,\dots,\ell\}$, of $H$ is non-discrete (here at least one of $I$ or $J$ is a proper subset).

We say $\Gamma$ is \emph{algebraically irreducible} if $\Gamma$ has no finite index subgroup splitting as the direct product of two infinite groups.

For uniform lattices the two definitions are equivalent by \cite[Theorem~4.2]{CapraceMonod2009b}; so we will simply refer to a lattice as \emph{irreducible} or \emph{reducible}.

\subsection{Graphs and complexes of lattices}
Let $\Gamma$ be a group and $K,L\leq \Gamma$ be subgroups.  If $L\cap K$ has finite index in $L$ and $K$ then we say $L$ and $K$ are \emph{commensurable}.  The \emph{commensurator} of $L$ in $\Gamma$ is the subgroup \[\Comm_\Gamma(L)\coloneqq \{g\in\Gamma\ |\ L^g\cap L \text{ has finite index in } L\text{ and } L^g\}.\]
If $\Comm_\Gamma(L)=\Gamma$ then we say $L$ is \emph{commensurated}.

Rather than recall the definitions and machinery from \cite{Hughes2021a} we will use it as a black box.  The key result for us is the following well known lemma.

\begin{lemma}\label{thm.structure} 
Let $X=X_1\times X_2$ be a proper cocompact minimal $\CAT(0)$ space and $H=\Isom(X_1)\times\Isom(X_2)$.  Suppose $X_1$ is a $\CAT(0)$ polyhedral complex.  Then, for any uniform $H$-lattice $\Gamma$, the cell stabilisers of $X_1$ in $\Gamma$ are commensurated, commensurable, and isomorphic to finite-by-$\{\Isom(X_2)$-lattices$\}$.
\end{lemma}

In our situation we will take $X_1$ to be a locally finite tree, or the universal cover of a Salvetti complex for a right-angled Artin group, and $X_2$ to be some proper irreducible cocompact minimal $\CAT(0)$ space.  The quotient space $X_1/\Gamma$ is endowed with a natural graph or complex of groups structure.  In the language of \cite{Hughes2021a} we call this data a \emph{graph} or \emph{complex of $\Isom(\EE^n)$-lattices}.  For example, every uniform $H$-lattice (where $H=\Aut(X_1)\times\Isom(\EE^n)$) splits as a graph or complex of commensurable finite-by-$\{n$-crystallographic$\}$ groups.

\subsection{Leary--Minasyan groups}\label{sec.LM}
The following groups were introduced in \cite{LearyMinasyan2019} by Leary and Minasyan as a class of groups containing the first examples of $\CAT(0)$ but not biautomatic groups; they were classified up to isomorphism by Valiunas \cite{Valiunas2020}.  In fact, they are not subgroups of any biautomatic group \cite{Valiunas2021a}.  Let $n \geq 0$, let $A \in \GL_n(\QQ)$, and let $L \leq \ZZ^n \cap A^{-1}(\ZZ^n)$ be a finite index subgroup. The group $\LM(A,L)$ is defined by the presentation
\begin{equation*}
\langle x_1,\ldots,x_n,t \mid [x_i,x_j]=1 \text{ for } 1 \leq i < j \leq n, t\mathbf{x}^{\mathbf{v}}t^{-1} = \mathbf{x}^{A\mathbf{v}} \text{ for } \mathbf{v} \in L \rangle,
\end{equation*}
where we write $\mathbf{x}^{\mathbf{w}} \coloneqq  x_1^{w_1} \cdots x_n^{w_n}$ for $\mathbf{w} = (w_1,\ldots,w_n) \in \ZZ^n$.  If $L$ is the largest subgroup of $\ZZ^n$ such that $AL$ is also a subgroup of $\ZZ^n$, then we denote $\LM(A,L)$ by $\LM(A)$.  We refer to the groups $\LM(A,L)$ and $\LM(A)$ as \emph{Leary--Minasyan groups}. The groups clearly split as HNN extensions $\ZZ^n\ast_L$.  The groups are $\CAT(0)$ if and only if $A$ is conjugate to an orthogonal matrix in $\GL_n(\RR)$ \cite[Theorem~7.2]{LearyMinasyan2019}.

As a concrete example, take \[A=\begin{bmatrix} 3/5 & -4/5\\ 4/5 & 3/5  \end{bmatrix} \text{ and } L=\left\langle\begin{bmatrix} 2\\ -1\end{bmatrix}, \begin{bmatrix} 1\\ 2\end{bmatrix} \right\rangle \text{ so } AL=\left\langle\begin{bmatrix} 2\\ 1\end{bmatrix}, \begin{bmatrix} -1\\ 2\end{bmatrix} \right\rangle.\]  Note that $L$ is index $5$ in $\ZZ^2$ and so must be a maximal subgroup.  It follows that
    \begin{equation}\label{eqn.LMA}
        \LM(A,L)=\LM(A)=\langle a,b, t\ |\ [a,b],\ ta^2b^{-1}t^{-1}=a^2b,\ tab^2t^{-1}=a^{-1}b^2\rangle. 
    \end{equation}

We say a matrix $A\in \GL_n(\RR)$ is \emph{irreducible} if no $A^k$ for $k\geq 1$ leaves invariant a proper nontrivial subspace of $\RR^n$. 

\begin{thm}\emph{\cite[Theorem~7.5]{LearyMinasyan2019}}
Suppose that $A$ has infinite order and is conjugate in $\GL_n(\RR)$ to an orthogonal matrix. Then, $\LM(A,L)$ is a uniform lattice in $\Isom(\EE^n)\times\Aut(\calt)$ whose projections to the factors are not discrete.  In particular, if $A$ is an irreducible matrix, then $\LM(A,L)$ is an irreducible lattice.
\end{thm}

We will detail the action on $\EE^2$ in the case of the Leary--Minasyan group \eqref{eqn.LMA}.  The group $\LM(A)$ has a representation $\pi$ to $\Isom(\EE^n)$ given by $\pi(a)=[1,0]^{\mathrm T}$, $\pi(b)=[0,1]^\mathrm{T}$, and $\pi(t)=A$.  Here $\mathrm{T}$ denotes transpose.  The matrix $A$ is a rotation by the irrational number $\cos^{-1}(3/5)$ and so has infinite order.  In particular, $\LM(A)$ is irreducible.

\section{Fibring lattices in a product of a tree and a Euclidean space}\label{sec.treeCase}
In this section we characterise irreducible $(\Isom(\EE^n)\times T)$-lattices as those which do not virtually $\mathsf{F}_1$-fibre (Theorem~\ref{thmx.fibring}).  Note that this result is new even for Leary-Minasyan groups.  Before we prove the theorem, we will collect some propositions.  Note that we do not need the full power of \Cref{prop.isomEnxX.H1} but only \Cref{claim.2}.  Thus, the reader who is not interested in cohomology computations may skip the spectral sequence argument.

\begin{prop}\label{prop.isomEnxX.H1}
Let $X$ be an irreducible locally finite $\CAT(0)$ polyhedral complex, let $A=\Aut(X)$ act cocompactly and minimally, and let $\Gamma$ be a uniform $(\Isom(\EE^n)\times A)$-lattice.  If $\Gamma$ is algebraically irreducible, then $H^1(\Gamma;\ZZ)\cong H^1(X/\Gamma;\ZZ)$.
\end{prop}

\begin{proof}[Proof of \Cref{prop.isomEnxX.H1}]
Let $\varphi\in H^1(\Gamma;\ZZ)=\hom(\Gamma,\ZZ)$, $P\coloneqq \pi_{\OO(n)}(\Gamma)$, and $N\coloneqq \ker(\pi_{\OO(n)})\triangleleft\Gamma$.  For the remainder of the proof an omission of coefficients in a (co)homology functor should be taken to mean coefficients with the trivial module $\ZZ$.



\begin{claim}\label{claim.2}
Let $L$ be a cell stabiliser in the action of $\Gamma$ on $X$.  Then, $\varphi|_L=0$.
\end{claim} 

\begin{claimproof}[Proof of claim:] Suppose for contradiction $\varphi$ is non-zero on some cell stabiliser $L$ of the $\Gamma$ action on $X$.  Then, after passing to a finite index subgroup of $L$, the restriction of $\varphi$ is non-zero on some subgroup isomorphic to $\ZZ^n$.  In particular, $\varphi$ defines a codimension $1$ subgroup $K$ of $\ZZ^n$ contained in $\ker(\varphi)$.  
Let $F\coloneqq \RR\otimes K\subset X\times\EE^n$ be the $(n-1)$-dimensional flat given by the flat torus theorem.
Now, for any $g\in\Gamma$ the flats $F$ and $g\cdot F$ are parallel.  Indeed, $g\cdot F$ is setwise stabilised by $K^g$, and $K^g\cap K$ has finite index in both $K$ and $K^g$.
Thus, $\Gamma$ fixes the one-dimensional subspace $F^\perp$.  But now, it follows that $\Gamma$ is a lattice in $\Isom(\EE^1)\times\Isom(\EE^{n-1})\times A$.  We have that the boundary $\partial (\EE^1\times \EE^{n-1}\times X)$ is equal to the join $S^0\aster S^{n-2}\aster \partial X$.  Now, $\Gamma$ has an index at most $2$ subgroup which fixes the $S^0$ factor.  It follows that $\Gamma$ must be reducible by \cite[Theorem~2(v)]{CapraceMonod2019} (see also the paragraph immediately following Theorem 2 in \emph{ibid.}), contradicting our hypothesis.  Thus, $\varphi|_L=0$.  
\end{claimproof}

Let $\Sigma^{(p)}$ be a representative set of orbits of $p$-cells for the action of $\Gamma$ on $X$.  The isomorphism will follow from a computation using the $\Gamma$-equivariant spectral sequence applied to the filtration of $X$ by skeleta (see \cite[Chapter~VII.7]{Brown1982}).  This spectral sequence takes the form
\[E_1^{p,q}\coloneqq \bigoplus_{\sigma\in\Sigma^{(p)}}H^q(\Gamma_\sigma)\Rightarrow H^{p+q}(\Gamma). \]
Since we are only interested in computing $H^1(\Gamma)$, the relevant part of the $E_1$-page is given by:
\[\begin{tikzpicture}
\matrix (m) [matrix of math nodes, nodes in empty cells, nodes={minimum  width=3ex, minimum height =3ex, outer sep =0pt}, column sep =6ex, row sep=2ex]{
p  &  &   &  &    \\
1    & \bigoplus_{\sigma\in\Sigma^{(0)}} H^1(\Gamma_\sigma)  & \bigoplus_{\sigma\in\Sigma^{(1)}} H^1(\Gamma_\sigma) &  &    \\
0   & \bigoplus_{\sigma\in\Sigma^{(0)}} H^0(\Gamma_\sigma)  & \bigoplus_{\sigma\in\Sigma^{(1)}} H^0(\Gamma_\sigma)  & \bigoplus_{\sigma\in\Sigma^{(2)}} H^0(\Gamma_\sigma) &   \\
  & 0  & 1  & 2 & q\\   };
 
 \path (m-3-1) -- (m-4-1) node[midway] (x1) {};
 \path (m-3-5) -- (m-4-5) node[midway] (x2) {};
 \draw[-stealth] (x1) -- (x2);
 
 \path (m-1-1) -- (m-1-2.west) node[midway, left=0.3cm] (y1) {};
 \path (m-4-1) -- (m-4-2.west) node[midway,left=0.3cm] (y2) {};
 \draw[-stealth] (y2) -- (y1);
 
 \draw[-stealth](m-3-2) -- (m-3-2-|m-3-3.west) node[midway,above]{$d_1^{0,0}$};
 \draw[-stealth](m-3-3) -- (m-3-3-|m-3-4.west) node[midway,above]{$d_1^{1,0}$};
 \draw[-stealth](m-2-2) -- (m-2-2-|m-2-3.west) node[midway,above]{$d_1^{0,1}$};
\end{tikzpicture}\]

Using the description of $d_1$ given in \cite[Chapter~VII.8]{Brown1982} it is easy to see that $E_2^{p,0}\cong H^p(X/\Gamma)$.  Now, the group $E_\infty^{0,1}$ is the image of the sum of restrictions \[\bigoplus_{\sigma\in\Sigma^{(0)}}\res^\Gamma_{\Gamma_\sigma}\colon H^1(\Gamma)\to\bigoplus_{\sigma\in\Sigma^{(0)}}H^1(\Gamma_\sigma)\]
and so must be $0$ by \Cref{claim.2}.  Also note for dimensional reasons $E_2^{0,0}=E_\infty^{0,0}$, $E_2^{1,0}=E_\infty^{1,0}$, $E_3^{0,1}=E_\infty^{0,1}$ and $E_3^{2,0}=E_\infty^{2,0}$.  Thus, the relevant part of the $E_\infty$-page is given by:
\[\begin{tikzpicture}
\matrix (m) [matrix of math nodes, nodes in empty cells, nodes={minimum  width=3ex, minimum height =3ex, outer sep =0pt}, column sep =6ex, row sep=2ex]{
p  &  &   &  &    \\
1    & 0  & E_\infty^{1,1} &  &    \\
0   & \ZZ  & H^1(X/\Gamma)  & E^{2,0}_\infty &   \\
  & 0  & 1  & 2 & q\\   };
  
 \path (m-3-1) -- (m-4-1) node[midway] (x1) {};
 \path (m-3-5) -- (m-4-5) node[midway] (x2) {};
 \draw[-stealth] (x1) -- (x2);
 
 \path (m-1-1) -- (m-1-2.west) node[midway, left=0.3cm] (y1) {};
 \path (m-4-1) -- (m-4-2.west) node[midway,left=0.3cm] (y2) {};
 \draw[-stealth] (y2) -- (y1);

\end{tikzpicture}\]
and so the desired isomorphism $H^1(\Gamma)\cong H^1(X/\Gamma)$ follows.
\end{proof}

We say a graph of groups $\calg$ is \emph{reduced}, if given an edge $e$ with distinct end points $v_1,v_2$, the inclusions $\Gamma_{e}\rightarrowtail\Gamma_{v_i}$ are proper. We say that a graph of groups $\calg$ is not an ascending HNN-extension if it is not an HNN-extension (it has more than one edge or more than one vertex), or it is an HNN-extension but both $\Gamma_e$ and $\Gamma_{\bar{e}}$ are proper subgroups of $\Gamma_v$.

We will need the following proposition of Cashen--Levitt \cite[Proposition~2.5]{CashenLevitt2016}.

\begin{prop}[Cashen--Levitt] \label{prop.CL}
Let $\Gamma$ be the fundamental group of a finite reduced graph of groups with $\Gamma$ finitely generated.  Assume that $\Gamma$ is not an ascending HNN-extension.  If $\varphi\in\Sigma^1(\Gamma)$, then $\varphi$ is non-trivial on every edge group.
\end{prop}

We are now ready to prove \Cref{thmx.fibring} from the introduction.

\begin{thmx}\label{thmx.fibring}
Let $\calt$ be a locally-finite leafless cocompact tree, not isometric to $\RR$, and let $T=\Aut(\calt)$.  Let $\Gamma$ be a uniform $(\Isom(\EE^n)\times T)$-lattice, then $\Gamma$ virtually algebraically fibres if and only if $\Gamma$ is reducible.
\end{thmx}

\begin{proof}
If $\Gamma$ is reducible, then $\Gamma$ virtually splits as $\ZZ\times\Gamma'$, where $\Gamma'$ is a $\CAT(0)$ group.  Hence, $\Gamma'$ is type $\mathsf{F}_\infty$.  In particular, $\Gamma$ virtually algebraically fibres.

We will now prove every irreducible uniform $(\Isom(\EE^n)\times T)$-lattice does not algebraically fibre, and this will prove the theorem since a finite index subgroup of an irreducible lattice is an irreducible lattice.  Now, suppose $\Gamma$ is an irreducible uniform $(\Isom(\EE^n)\times T)$-lattice.  By \Cref{thm.structure}, the group $\Gamma$ splits as a graph of $\Isom(\EE^n)$-lattices, and so is the fundamental group of a graph of groups with vertex and edge stabilisers finite-by-$\{\Isom(\EE^n)$-lattices$\}$.  

\begin{claim}$\Gamma$ splits as a reduced graph of groups and is not an ascending HNN extension.\end{claim}

\begin{claimproof}[Proof of Claim:] We may assume the graph of groups is reduced by contracting any edges with a trivial amalgam $L\ast_L K$.  Note that these contractions do not change the vertex and edge stabilisers, but may change the Bass-Serre tree (the tree will still not be quasi-isometric to $\RR$ since there are necessarily other vertices of degree at least $3$).  Since we only contract finitely many orbits of edges in a finite valence tree, the resulting tree remains locally finite.  To see the lattice remains irreducible note: First, that the projection to $\Isom(\EE^n)$ remains unchanged. Second, that the vertex stabilisers of $\Gamma$ are non-discrete in both trees.

Now for $\Gamma$ to be an ascending HNN-extension the graph $\calt/\Gamma$ must consist of a single vertex and edge.  Let $t$ be the stable letter of $\Gamma$, then $t$ acts as an isometry on $\calt\times\EE^n$ and so for any vertex stabiliser $\Gamma_v$ of $\Gamma$ acting on $\calt$, the actions of $\Gamma_v$ and $\Gamma_v^t$ on $\EE^n$ have the same covolume.  Now, covolume is multiplicative when passing to covers. In particular, under the projection $\pi_{\Isom(\EE^n)}$, the two embeddings of the projection of the edge group $\Gamma_e$ into the projection of the vertex group $\Gamma_v$ must have the same index.  Now, if $\pi_{\Isom(\EE^n)}(t)$ (virtually) centralised $\pi_{\Isom(\EE^n)}(\Gamma_v)$, then the projection $\pi_{\Isom(\EE^n)}(\Gamma)$ would be virtually abelian.  But, in this case $\Gamma$ is reducible by \cite[Theorem~2(iii)]{CapraceMonod2019}.   Thus, the two embeddings of $\pi_{\Isom(\EE^n)}(\Gamma_e)$ into the vertex group $\pi_{\Isom(\EE^n)}(\Gamma_v)$ must both have index at least $2$, yielding the claim. 
\end{claimproof}

Now, by $H^1(\Gamma;\ZZ)\otimes\RR\cong H^1(\Gamma;\RR)$ and \Cref{claim.2}, for every character $\phi\in H^1(\Gamma;\RR)$ we see that $\phi$ restricted to a vertex or edge group is zero.  Since $\Gamma$ is the fundamental group of a reduced graph of groups, is not an ascending HNN extension, and $\phi$ vanishes on every edge group, we may apply \Cref{prop.CL} to deduce that $\phi\not\in\Sigma^1(\Gamma)$.  Hence, $\Gamma$ does not (virtually) $\mathsf{F}_1$-fibre.
\end{proof}

\begin{lemma}\label{lem.reducible.E2}
    A reducible uniform lattice in $\Isom(\EE^2)\times T$ is virtually $F_m\times\ZZ^2$ for some $m\geq 2$.
\end{lemma}
\begin{proof}
    Since $\Gamma$ is reducible, by definition (at least) one of three cases holds for $\Gamma$.
    
    \noindent \textbf{Case 1:} \emph{The projection to $\Isom(\EE^2)$ is discrete}. Hence, virtually abelian.  By \cite[Theorem~2(iii)]{CapraceMonod2019}, the group $\Gamma$ splits as $\ZZ^2\times\Gamma'$ where $\Gamma'$ is a uniform $T$-lattice.  Whence, the claim.

    \noindent \textbf{Case 2:} \emph{The projection to $T$ is discrete}.  In this case $\Gamma.\Isom(\EE^2)$ is closed in $\Isom(\EE^2)\times T$ and so $\Gamma_E=\Gamma\cap\Isom(\EE^2)$ is a lattice.  But, now $\pi_{\Isom(\EE^2)}(\Gamma)$ normalises $\Gamma_E$ and so must be a virtually abelian subgroup of $\Isom(\EE^2)$.  We now conclude as in Case 1.

    \noindent \textbf{Case 3:} \emph{The group $\Gamma$ virtually preserves a product decomposition $\EE^1\times\EE^1$.}  In this case $\Gamma$ is virtually a subgroup of $\Isom(\EE^1)\times \Isom(\EE^1)\times T$ and so $\pi_{\Isom(\EE^2)}(\Gamma)$ is virtually abelian.  We now conclude as in Case 1.
\end{proof}

\begin{corx}\label{corx.fibring.dim2}
With notation as in Theorem~\ref{thmx.fibring} suppose $n=2$.  Then $\Gamma$ virtually fibres if and only if $\Gamma$ is reducible.
\end{corx}
\begin{proof}
This follows from Theorem~\ref{thmx.fibring} and \Cref{lem.reducible.E2}.  
\end{proof}

\medskip

\section{Lattices in symmetric spaces and Salvetti complexes}\label{extending actions}

\subsection{Right-angled Artin groups}
Let $L$ be a flag complex.  We begin with a wedge of circles (made of a single vertex and edge), one for each edge of $L$, attached along a common vertex $x$.  For each edge $\{v,w\}$ in $L$ we attach a $2$-torus along the word $vwv^{-1}w^{-1}$.  Continuing inductively, for each $n\geq 2$ cell $\{v_1,\dots,v_n\}$ in $L$ we attach an $n$-torus such that the faces correspond to already attached $(n-1)$-tori.  We denote the resulting space by $X_L$ and call it the \emph{Salvetti complex} of $L$.

The fundamental group $A_L\coloneqq \pi_1(X_L)$ is the \emph{right-angled Artin group} (RAAG) on $L$.  The group has a generating set given by the vertices of $L$, \emph{the standard generators}, and the relations that two generators $v$ and $w$ commute if and only if they are joined by an edge in $L$.  The Salvetti complex $X_L$ has universal cover $\widetilde{X}_L$ which is the quintessential example of a $\CAT(0)$ cube complex.  We endow the edges of $\widetilde X_L$ with a labelling and orientation given by the labels of vertices of $L$ and call this the \emph{standard labelling and orientation}.  We will denote the isometry group of $\widetilde{X}_L$ by $S_L$ and endow it with the topology given by uniform convergence on compacta.  Note that we do not require $S_L$ to preserve the standard labelling or orientation.  We say a RAAG is \emph{irreducible} if it does not split as the direct product of two infinite subgroups.

\subsection{The plan}
Let $L$ be a flag complex on $m\geq 3$ vertices such that $L$ is not a non-trivial join.  Let $\calt_n$ denote the $n$-regular tree and let $T_n$ denote its isometry group.

We start with an irreducible uniform lattice $\Gamma$ in $H\times T_3$ where $H$ is a non-discrete isometry group of an irreducible $\CAT(0)$ space admitting a uniform lattice.  As an example one could take an irreducible uniform lattice in $\PSL_2(\RR)\times \PSL_2(\QQ_2)$.  Here $\PSL_2(\QQ_2)<T_3$ because $\PSL_2(\QQ_2)$ acts faithfully by isometries on its Bruhat--Tits tree, which in this case is the $3$-regular tree.  Such a lattice may be constructed as the $\ZZ[\frac{1}{2}]$ points of some $\SO(2,1;Q)$, where $Q$ is a quadratic form.

Our scheme is now as follows: First, using $\Gamma$, we will obtain an irreducible uniform lattice $\widetilde\Gamma \leqslant H\times T_4$ which acts on $\calt_4$ in way that preserves the standard edge labelling when thought of as the Cayley graph of $F_2$.  Note that this is the \emph{standard labelling} of $\calt_4$.  Second, we will obtain an irreducible uniform lattice $\Gamma_L\leqslant H\times S_L$ where $S_L=\Aut(\widetilde X_L)$ and $\widetilde X_L$ is the universal cover of the Salvetti complex $X_L$ that preserves the standard labelling away from edges `coming from' $\calt_4$.

\subsection{From a 3-valent to a 4-valent tree}

The following proposition achieves the first step of our scheme.  It is essentially due to Jingyin Huang (see \cite[Lemma~9.2]{Huang2018}) although we have taken the liberty to express it in a more general setting for our purposes.

\begin{prop}\label{prop.3val-4val}
   Let $H$ be a non-discrete isometry group of a proper irreducible cocompact minimal $\CAT(0)$ space. Let $\Gamma\leqslant H\times T_3$ be an irreducible uniform lattice.  Then, there exists an irreducible uniform lattice $\widetilde \Gamma$ in $H\times T_4$ such that
   \begin{enumerate}
       \item $\Gamma$ preserves the standard labelling of $\calt_4$;
       \item there is a short exact sequence $1\to F_\infty \to \widetilde \Gamma \to \Gamma \to 1$. \label{prop.3val-4val.ses}
   \end{enumerate}
\end{prop}

Rather than reproduce the proof, we briefly sketch the construction.  Given the $3$-regular tree $\calt_3$ we replace every edge with a pair of edges labelled $a$ and orient them in a compatible way.  We replace every vertex with an oriented $3$-cycle with each edge labelled $b$ (see \Cref{fig:3val-4val}).  The resulting graph $\calg$ is $4$-valent and any group action on $\calt_3$ gives an action on $\calg$ that is label preserving and preserves the orientation of $a$ edges.  The universal cover of $\calg$ is the $4$-regular tree $\calt_4$ so we may lift $\Gamma$ to a new group $\widetilde \Gamma$ which is an extension of $\Gamma$ by $\pi_1\calg=F_\infty$.  Our new group $\widetilde \Gamma$ clearly preserves the labelling of $\calt_4$ and the orientation of $a$-edges.  To see that $\widetilde \Gamma$ is an irreducible lattice in $H\times T_4$ note that the projection $\widetilde\Gamma\to H$ is the composition $\widetilde \Gamma\onto \Gamma \to H$ and that the set of elements acting non-trivially in a vertex stabiliser of the action on $\calt_3$ and $\calt_4$ is infinite (so both projections are non-discrete).


\begin{figure}[h]
    \centering
    \[\begin{tikzpicture}
        \node[circle,fill=black,inner sep=0pt,minimum size=5pt] at (0,0) (P1) {};
        \node[circle,fill=black,inner sep=0pt,minimum size=5pt] at (-1,1) (P2) {};
        \node[circle,fill=black,inner sep=0pt,minimum size=5pt] at (-1,-1) (P3) {};
        \node[circle,fill=black,inner sep=0pt,minimum size=5pt] at (2,1) (P5) {};
        \node[circle,fill=black,inner sep=0pt,minimum size=5pt] at (1,0) (P4) {};
        \node[circle,fill=black,inner sep=0pt,minimum size=5pt] at (2,-1) (P6) {};
        \draw (P1) -- (P2); \draw (P1) -- (P3);  \draw (P1) -- (P4); 
        \draw (P4) -- (P5); \draw (P4) -- (P6);
        
        \node at (3,0) {$\longrightarrow$};
        
        \node[circle,fill=black,inner sep=0pt,minimum size=5pt] at (5.5,0.33) (Q1a) {};
        \node[circle,fill=black,inner sep=0pt,minimum size=5pt] at (5.5,-0.33) (Q1c) {};
        \node[circle,fill=black,inner sep=0pt,minimum size=5pt] at (6,0) (Q1b) {};
        
        \node[circle,fill=black,inner sep=0pt,minimum size=5pt] at (4.33,1.33) (Q2) {};
        \node[circle,fill=black,inner sep=0pt,minimum size=5pt] at (4.33,-1.33) (Q3) {};

        \node[circle,fill=black,inner sep=0pt,minimum size=5pt] at (7.5,0) (Q4a) {};
        \node[circle,fill=black,inner sep=0pt,minimum size=5pt] at (8,0.33) (Q4b) {};
        \node[circle,fill=black,inner sep=0pt,minimum size=5pt] at (8,-0.33) (Q4c) {};
        
        \node[circle,fill=black,inner sep=0pt,minimum size=5pt] at (8.66+0.5,1.33) (Q5) {};
        \node[circle,fill=black,inner sep=0pt,minimum size=5pt] at (8.66+0.5,-1.33) (Q6) {};
        
        \path [black, bend left, -stealth] (Q1a) edge node[midway, left] {$a$} (Q2); 
        \path [black, bend left, -stealth] (Q2) edge node[midway, above] {$a$} (Q1a);
        \path [black, bend left, -stealth] (Q1c) edge node[midway, below] {$a$} (Q3); 
        \path [black, bend left, -stealth] (Q3) edge node[midway, left] {$a$} (Q1c); 

        \path [black, bend left, -stealth] (Q1a) edge node[midway, above] {$b$} (Q1b); 
        \path [black, bend left, -stealth] (Q1b) edge node[midway, below] {$b$} (Q1c); 
        \path [black, bend left, -stealth] (Q1c) edge node[midway, left] {$b$} (Q1a); 
        
        \path [black, bend left, -stealth] (Q1b) edge node[midway, above] {$a$} (Q4a); 
        \path [black, bend left, -stealth] (Q4a) edge node[midway, below] {$a$} (Q1b);

        \path [black, bend left, -stealth] (Q4a) edge node[midway, above] {$b$} (Q4b); 
        \path [black, bend left, -stealth] (Q4b) edge node[midway, right] {$b$} (Q4c); 
        \path [black, bend left, -stealth] (Q4c) edge node[midway, below] {$b$} (Q4a); 

        \path [black, bend left, -stealth] (Q4b) edge node[midway, above] {$a$} (Q5); 
        \path [black, bend left, -stealth] (Q5) edge node[midway, below] {$a$} (Q4b);
        \path [black, bend left, -stealth] (Q4c) edge node[midway, above] {$a$} (Q6); 
        \path [black, bend left, -stealth] (Q6) edge node[midway, below] {$a$} (Q4c); 
    \end{tikzpicture}\]
    \caption{An illustration of Jingyin Huang's trick to turn a $3$-regular tree into a $4$-regular graph.}
    \label{fig:3val-4val}
\end{figure}
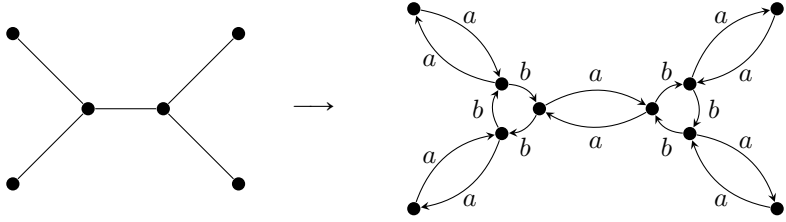

\begin{lemma}\label{lem.vtf1}
If $\Gamma$ is (virtually) torsion-free, then so is $\widetilde\Gamma$.
\end{lemma}
\begin{proof}
    This follows from the short exact sequence in \Cref{prop.3val-4val}\eqref{prop.3val-4val.ses}.
\end{proof}

\subsection{From a 4-valent tree to a Salvetti complex}

The following proposition is step two of our scheme.  It is a special case of a more general construction (see~\cite[Theorem~7.4]{Hughes2021a}).  Note that this result was known to Jingyin Huang in the case that $H$ is the automorphism group of a $3$-regular tree \cite[Theorem~9.5]{Huang2018} and a related construction has appeard in the work of Camille Horbez and Jingyin Huang \cite[Proposition~4.6]{HorbezHuang2022}.

\begin{prop}\label{prop.4val-Salv}
    Let $L$ be a flag complex on $m\geq 3$ vertices such that there are $2$ vertices labelled $a$ and $b$ not joined by an edge and such that $L$ is not a non-trivial join.  Let $H$ be a non-discrete isometry group of a proper irreducible cocompact minimal $\CAT(0)$ space admitting a uniform lattice. Let $\Gamma\leqslant H\times T_4$ be an irreducible uniform lattice preserving the standard labelling $\{a,b\}$ of $\calt_4$.  Then, there exists an irreducible uniform lattice $\Gamma_L$ in $H\times S_L$ which preserves the standard labelling for all edges of $\widetilde X_L$ and the standard orientation on $\widetilde X_L$ except for edges labelled $a$ or $b$.
\end{prop}

We sketch the construction of the proposition for the benefit of the reader.  Let $V=\{a,b\}\subset VL$ be a proper subset.  Consider $\pi\colon A_L\twoheadrightarrow A_V=F_2$ given by $a\mapsto a$, $b\mapsto b$ and $v\mapsto 1$ for any $v\in VL\backslash V$.  This has kernel $\ker(\pi)$ and covering space $\widetilde Y\to X_L$.  We may identify the vertex set of $\calt_{4}$ with the vertex set of $Y$ via the embedding of $\calt_{4}\rightarrowtail Y$ given by `unwrapping' the $S^1\bigvee S^1 =X_V\subset X_L$ corresponding to the vertices $a$ and $b$.  The $1$-skeleton $Y^{(1)}$ of $Y$ is obtained from $\calt_{4}$ by attaching to each vertex of $\calt_{4}$ a circle for each $v\in VL$ not equal to $a$ or $b$.

Now, $\Gamma$ acts by label-preserving isometries on $\calt_{4}$.  It follows $\Gamma$ acts by label-preserving isometries on $Y$ and preserves the standard orientation of any edge with label not equal to $a$ or $b$ (see \cite[Proposition~4.6]{HorbezHuang2022}).  The group $\Gamma_L$ is then defined to be the lifts of all elements of $\Gamma$ to $S_L$.  That is, we have a short exact sequence
\begin{equation}\label{ses.4val-Salv}\begin{tikzcd}
1 \arrow[r] & \ker(\pi) \arrow[r] & \Gamma_L \arrow[r] & \Gamma \arrow[r] & 1.
\end{tikzcd} \end{equation}

The hypothesis on $L$ ensure that $\widetilde X_L$ is an irreducible $\CAT(0)$ space.  Thus, to see $\Gamma_L$ is irreducible it is enough to notice that the projection to $H$ is exactly the composition $\Gamma_L\onto \Gamma\to H$ and that the vertex stabilisers of $\Gamma$ acting on $\calt_4$ are also vertex stabilisers for $\Gamma_L$ acting on $\widetilde X_L$.  In particular, since they were non-discrete when projected to $T_4$, they are non-discrete when projected to $S_L$.

\begin{remark}\label{rem.orientable.2}
    If the set-wise stabiliser of an edge $e$ of $\calt_4$ in $\Gamma$ equals its point-wise stabiliser, then the same holds true for a lift of $e$ in $\widetilde X_L$ and its stabiliser in $\Gamma_L$.
\end{remark}

\begin{lemma}\label{lem.vtf2}
    If $\Gamma$ is (virtually) torsion-free, then so is $\Gamma_L$.
\end{lemma}
\begin{proof}
    This follows immediately from the short exact sequence \eqref{ses.4val-Salv}.
\end{proof}

\section{Fibrations over the circle}\label{sec.irr.fibring}

\subsection{Filtrations}
Let $\Gamma$ be a group and $X$ be a $\Gamma$-CW complex.  We say $X$ is \emph{$n$-good} if
\begin{enumerate}
    \item $X$ is $n$-acyclic, i.e. $\widetilde{H}_k(X)=0$ for $k\leq n$;
    \item for $0\leq p \leq n$, the stabiliser $\Gamma_\sigma$ of any $p$-cell $\sigma$ is of type $\rm{FP}_{n-p}$.
\end{enumerate}
A \emph{filtration} of $X$ is a family $\{X_\alpha\}_{\alpha\in I}$ of $\Gamma$-invariant subcomplexes such that $I$ is a directed set, $X_\alpha\subseteq X_\beta$ when $\alpha\leq\beta$, and $X=\bigcup_{\alpha}X_\alpha$.  The filtration is of \emph{finite $n$-type} if the $X_\alpha/\Gamma$ have finite $n$-skeleton.  We say that $\{X_\alpha\}$ is \emph{$\widetilde{H}_k$-essentially trivial} (resp. \emph{$\pi_k$-essentially trivial}) if for each $\alpha$ there is $\beta\geq\alpha$ such that $\widetilde{H}_k(\ell_{\alpha,\beta})=0$ (resp. $\pi_k(\ell_{\alpha,\beta})=0$), where $\ell_{\alpha,\beta}\colon X_\alpha\rightarrowtail X_\beta$ is the inclusion.

We will make use of the two criteria due to Brown.

\begin{thm}\emph{\cite{Brown1987}} \label{browns.crit.classic.fp}
Let $X$ be an $n$-good $\Gamma$-complex with a filtration $\{X_\alpha\}$ of finite $n$-type.  Then $\Gamma$ is of type $\mathsf{FP}_n$ if and only if the directed system $\{X_\alpha\}$ is $\widetilde{H}_k$-essentially trivial for all $k<n$.\hfill\qed
\end{thm}

\begin{thm}\emph{\cite{Brown1987}} \label{browns.crit.classic.f}
Let $X$ be a simply connected $\Gamma$-complex such that the vertex stabilisers are finitely presented and the edge stabilisers are finitely generated.  Let $\{X_\alpha\}$ be a filtration of $X$ of finite $2$-type and let $v\in\bigcap X_\alpha$ be a basepoint.  If $\Gamma$ is finitely generated, then $\Gamma$ is finitely presented if and only if $\{(X_\alpha,v)\}$ is $\pi_1$-essentially trivial.\hfill\qed
\end{thm}

\subsection{Computations}
Let $L$ be a flag complex on $m\geq 3$ vertices such that at least $2$ vertices $a,b$ are not joined by an edge and such that $L$ is not a non-trivial join.

Let $H$ be a non-discrete isometry group of a proper irreducible cocompact minimal $\CAT(0)$ space admitting a uniform lattice. Let $\Gamma\leqslant H\times T_3$ be an irreducible uniform lattice.  Let $\Gamma_L$ denote the lattice in $H\times S_L$ obtained by sequentially applying \Cref{prop.3val-4val} and \Cref{prop.4val-Salv}.  Note that $\Gamma_L$ preserves the standard labelling of $\widetilde X_L$ and by \Cref{rem.orientable.2} preserves the standard orientation of edges of $\widetilde X_L$ except those labelled by $a$ or $b$. 

Recall $A_L$ is the RAAG on $L$.  The standard generators of $A_L$ determine, on $\widetilde X_L$, an edge labelling $E(\widetilde{X}_L)\to\calv\coloneqq \{v_1,\dots, v_{m-2},a,b\}$.  Here, $a$ and $b$ are the vertices in the construction of $\Gamma_L$.   Each character $\phi\colon A_L\to \RR$ determines a height function on $\widetilde X_L$ which can be described by identifying the Cayley graph of $A_L$ with the $1$-skeleton of $\widetilde X_L$ and assigning real numbers to the edge labels $\calv$, thought of as signed lengths (with respect to the standard orientation of $\widetilde X_L$).  

\begin{claim}
    If $\phi$ vanishes on $a$ and $b$, then $\phi$ determines an element $\psi\in\hom(\Gamma_L,\RR)$.  Moreover, if $\phi\in\hom(A_L;\ZZ)$, then $\psi\in \hom(\Gamma_L;\ZZ)$.
\end{claim} 
\begin{claimproof}[Proof of Claim]


The action of $\Gamma_L$ on $\widetilde X_L$ preserves the labelling of edges and orientation of edges that are not labelled by $a$ or $b$, so it preserves the signed lengths of edge paths in $\widetilde X_L$.  Hence, it preserves the level in $\widetilde X_L$ for the height function induced by $\phi$, and it preserves the signed distances between level sets.  The action of $\Gamma_L$ on the collection of level sets induces $\psi\in\hom(\Gamma_L;\RR)$.  The moreover is clear.
\end{claimproof}

Suppose now that $\phi$ (and hence $\psi$) is an integral character.  The characters induce a height function $h\colon\widetilde X_L\to \RR$ with vertices taking integer values. The explicit details of this are not needed so we defer the interested reader to \cite{BuxGonzalez1999}, specifically Remark 10.  The important part for us is that both groups $A_L$ and $\Gamma_L$ act on $\widetilde{X}_L$ cocompactly and either freely in the first case or with $\CAT(0)$ stabilisers in the second case (see \Cref{thm.structure}).  Both $\ker(\phi)$ and $\ker(\psi)$ act cocompactly on level sets of the induced height function $h$.  Here a level set is the preimage of a compact connected subset of $\RR$.

Now, the stabilisers of cells in the level set are exactly the stabilisers of $\Gamma_L$ that fix a cell in the level set.  Indeed, these subgroups of $\Gamma_L$ are elliptic on $\widetilde X_L$, and so by definition of $\psi$ are contained in $\ker(\psi)$.  It follows that the stabilisers of the action on the level sets are type $\mathsf{F}_\infty$ (in fact they are type $\mathsf{F}$ if $\Gamma_L$ is torsion-free). Thus, the hypotheses of Brown's criteria are satisfied for both $\ker(\varphi)<A_L$ and $\ker(\psi)<\Gamma_L$ when acting on level sets of the height function $h$.  We have almost proved the following theorem:

\begin{thm}\label{thm.fibring.construction}
    Let $L$ be a flag complex with vertices $\{v_1,\dots,v_{n-2},a,b\}$ and suppose $a$ and $b$ do not span an edge. 
 Let $\phi\in H^1(A_L;\ZZ)$ vanish on $a$ and $b$ and let $\psi$ be the corresponding character in $H^1(\Gamma_L;\ZZ)$.  
 Then, $\ker(\phi)$ is type $\mathsf{FP}_n$ (resp. type $\mathsf{F}_n$) if and only if  $\ker(\psi)$ is type $\mathsf{FP}_n$ (resp. type $\mathsf{F}_n$).  
 Moreover, if $\Gamma_L$ is torsion-free then $\ker(\phi)$ is type $\mathsf{FP}$ (resp. $\mathsf{F}$) if and only if  $\ker(\psi)$ is type $\mathsf{FP}$ (resp. type $\mathsf{F}$).
\end{thm}
\begin{proof}
    The case of $\mathsf{FP}_n$ follows from \Cref{browns.crit.classic.fp}, the case of $\mathsf{F}_n$ follows from \Cref{browns.crit.classic.f} and the classical fact that a finitely presented group being type $\mathsf{FP_n}$ is in fact type $\mathsf{F}_n$ \cite[VIII.7 Ex. 1]{Brown1982}. We now prove the moreover.  If $\Gamma_L$ is torsion-free, then $\cd(\Gamma_L)$ is finite, hence so is $\cd(\ker(\psi))$. Suppose $\ker(\phi)$ is type $\mathsf{FP}$. Then, $\ker(\psi)$ is type $\mathsf{FP}_n$ for all $n$, that is type $\mathsf{FP}_\infty$.  Thus, by \cite[VIII(6.1)]{Brown1982}, $\ker(\psi)$ is type $\mathsf{FP}$.  Finally, if $\ker(\phi)$ is type $\mathsf{F}$, then $\ker(\psi)$ is type $\mathsf{FP}$, type $\mathsf{F}_n$ for all $n$, and admits a finitely dominated Eilenberg--Maclane space \cite[VIII(6.1)]{Brown1982}.  It remains to show that Wall's finiteness obstruction in $\widetilde K_0(\ZZ \ker\psi)$ vanishes \cite{Wall1965,Wall1966}.  Now, since $\CAT(0)$ groups satisfy the Farrell--Jones Conjecture (in dimension $0$) \cite{BartelsLuck2012} and the conjecture is closed under taking subgroups (loc. cit.), we see that $\ker(\psi)$ satisfies the Farrell--Jones Conjecture (in dimension $0$) .  In particular, as $\ker(\psi)$ is torsion-free, we have $\widetilde K_0(\ZZ \ker\psi)=0$.  Hence, $\ker(\psi)$ admits a finite classifying space as required.  In each case the converse follows by reversing the argument.
\end{proof}

The previous theorem is easy to apply because the BNSR invariants of RAAGs are known \cite{BestvinaBrady1997,MeierMeinertVanWyk1998,BuxGonzalez1999}.  We reproduce the result here for the convenience of the reader.

Let $L$ be a flag complex with RAAG $A_L$.  Each vertex of $L$ corresponds to a standard generator of $A_L$.  Given a character $\phi\colon A_L\to\RR$, let $L^\dagger$ denote the full subcomplex of $L$ spanned by vertices $v$ such that $\phi(v)=0$, and let $L^\ast$ denote the full subcomplex of $L$ spanned by vertices $v$ such that $\phi(v)\neq0$.

\begin{thm}[Bestvina--Brady, Meier--Meinert--VanWyk, Bux--Gonzalez]\label{thm.BBL.crit}
Let $L$ be a flag complex and let $\phi\in H^1(A_L;\RR)$.  The following are equivalent:
\begin{enumerate}
    \item $\ker\phi$ is type $\mathsf{FP}_{n+1}(\Z)$, resp. $\ker\phi$ is type $\mathsf{F}_{n+1}$;
    \item For every (possibly empty) dead simplex $\sigma\in L^\dagger$ the living link $\Lk_{L^\ast}(\sigma)\coloneqq L^\ast\cap\Lk_L(\sigma)$ is $(n-\dim(\sigma)-1)$-acyclic, resp. $L^\ast$ is, additionally, $n$-connected.
\end{enumerate}
\end{thm}

Applying this we obtain the result promised by the title.

\begin{thmx}\label{thmx.C}
    There exists an irreducible lattice fibring over the circle.
\end{thmx}
\begin{proof}
    Pick a torsion-free irreducible uniform lattice $\Gamma$ in $\PSL_2(\RR)\times\PSL_2(\QQ_2)$ and let $L$ be the following flag complex (the triangles are filled in with $2$-cells):
    \begin{center}
    \begin{tikzpicture}
        \node[circle,fill=black,inner sep=0pt,minimum size=5pt] at (0,0) (P1) {};
        \node[circle,fill=black,inner sep=0pt,minimum size=5pt] at (1,0) (P2) {};
        \node[circle,fill=black,inner sep=0pt,minimum size=5pt] at (0,1) (P4) [label=left:{$a$}] {};
        \node[circle,fill=black,inner sep=0pt,minimum size=5pt] at (1,1) (P3) {};
        \node[circle,fill=black,inner sep=0pt,minimum size=5pt] at (2,0) (Q1) {};
        \node[circle,fill=black,inner sep=0pt,minimum size=5pt] at (3,0) (Q2) [label=right:{$b$}] {};
        \node[circle,fill=black,inner sep=0pt,minimum size=5pt] at (2,1) (Q4)  {};
        \node[circle,fill=black,inner sep=0pt,minimum size=5pt] at (3,1) (Q3) {};
        \draw (P1) -- (P2); 
        \draw (P2) -- (P3); 
        \path (P3) edge node[midway, above] {} (P4) ; 
        \path (P4) edge node[midway, left] {} (P1); 
        \draw (P1) -- (P3);
        \draw (P2) -- (Q1); \draw (Q1) -- (Q4); \draw (Q4) -- (P3); \draw (P2) -- (Q4);
        \path (Q1) edge node[midway, below] {}  (Q2); 
        \path (Q2) edge node[midway, right] {}  (Q3); 
        \draw (Q3) -- (Q4); 
        \draw (Q4) -- (Q1); 
        \draw (Q1) -- (Q3);
    \end{tikzpicture}
    \end{center}
    Let $\phi\colon A_L\onto\Z$ be defined by mapping each standard generator to $1$ except for $a$ and $b$ which map to $0$. 
    Applying \Cref{prop.3val-4val}, \Cref{prop.4val-Salv}, and \Cref{thm.fibring.construction} sequentially we obtain an irreducible uniform lattice $\Gamma_L<\PSL_2(\RR)\times S_L$ and a character $\psi\colon\Gamma_L\to\Z$.
     It is easy to see that $L^\ast$ is the full subcomplex of $L$ spanned by unlabelled vertices and is clearly contractible.  Moreover, both $L^\ast\cap\Lk_L(a)$ and $L^\ast\cap\Lk_L(b)$ are both equal to an edge,  hence, they are contractible as well. 
    By \Cref{thm.BBL.crit}, $\ker(\phi)$ is of type $\mathsf{F}$ (one needs to run the same argument as in \Cref{thm.fibring.construction}).  Thus, as $\Gamma_L$ is torsion-free (\Cref{lem.vtf1} and \Cref{lem.vtf2}), we have by \Cref{thm.fibring.construction} that $\ker(\psi)$ is also of type $\mathsf{F}$.
\end{proof}


\bibliographystyle{halpha}
\bibliography{refs.bib}

\end{document}